\documentclass{article}%
\usepackage{amsfonts}
\usepackage{amsmath}
\usepackage{amssymb}
\usepackage{graphicx}
\usepackage[colorlinks=true, pdfstartview=FitV, linkcolor=blue,
citecolor=blue, urlcolor=blue]{hyperref}
\usepackage{color}%
\providecommand{\U}[1]{\protect\rule{.1in}{.1in}}
\newtheorem{theorem}{Theorem}

\newtheorem{corollary}[theorem]{Corollary}

\newtheorem{lemma}[theorem]{Lemma}

\newtheorem{proposition}[theorem]{Proposition}

\newenvironment{proof}[1][Proof]{\noindent\textbf{#1.} }{\ \rule{0.5em}{0.5em}}

\begin{document}

\title{Recurrences and Congruences for Higher order Geometric Polynomials and Related Numbers}
\author{Levent Karg\i n\thanks{leventkargin48@gmail.com} and Mehmet
Cenkci\thanks{cenkci@akdeniz.edu.tr}\\Department of Mathematics, Akdeniz University, Antalya Turkey}
\maketitle

\begin{abstract}
We obtain new recurrence relations, an explicit formula, and convolution
identities for higher order geometric polynomials. These relations generalize
known results for geometric polynomials, and lead to congruences for higher
order geometric polynomials, particularly for $p$-Bernoulli numbers.

\end{abstract}

\section{Introduction}

\setcounter{theorem}{0} \setcounter{equation}{0}

For a complex variable $y$, the geometric polynomials $w_{n}(y)$ of degree $n$
are defined by \cite{T}
\begin{equation}
w_{n}\left(  y\right)  =\sum_{k=0}^{n}%
\genfrac{\{}{\}}{0pt}{}{n}{k}%
k!y^{k}, \label{30}%
\end{equation}
where $%
\genfrac{\{}{\}}{0pt}{}{n}{k}%
$ is the Stirling number of the second kind \cite{Graham}. These polynomials
have been studied from analytic, combinatoric, and number theoretic points of
view. Analytically, they are used in evaluating geometric series of the form
\cite{B}
\[
\sum_{k=0}^{\infty}k^{n}y^{k},
\]
with
\[
\left(  y\frac{d}{dy}\right)  ^{n}\frac{1}{1-y}=\sum_{k=0}^{\infty}k^{n}%
y^{k}=\frac{1}{1-y}w_{n}\left(  \frac{y}{1-y}\right)  .
\]
for every $\left\vert y\right\vert <1$ and every $n\in{\mathbb{Z}}$, $n\geq0$.
Combinatorially, they are related to the total number of preferential
arrangements of $n$ objects
\[
w_{n}\left(  1\right)  :=w_{n}=\sum_{k=0}^{n}%
\genfrac{\{}{\}}{0pt}{}{n}{k}%
k!,
\]
that is, the number of partitions of an $n$-element set into $k$ nonempty
distinguishable subsets (c.f. \cite{Dasef-Kautz-1997}). Number theoretic
studies on the geometric polynomials are mostly originated from their
exponential generating function
\[
\sum_{n=0}^{\infty}w_{n}\left(  y\right)  \frac{t^{n}}{n!}=\frac{1}{1-y\left(
e^{t}-1\right)  }.
\]
For example, setting $y=-\frac{1}{2}$ gives
\[
w_{n}\left(  -\frac{1}{2}\right)  =\frac{2}{n+1}\left(  1-2^{n+1}\right)
B_{n+1}=-\frac{T_{n}}{2^{n}},
\]
where $B_{n}$ are Bernoulli numbers and $T_{n}$ are tangent numbers. Bernoulli
numbers also occur in integrals involving geometric polynomials, namely we
have \cite{Keller}
\begin{equation*}%
{\displaystyle\int\limits_{0}^{1}}
w_{n}\left(  -y\right)  dy=B_{n},\text{ \ \ \ }n>0. \label{24}%
\end{equation*}
Moreover, we note that \cite{Kargin2}
\begin{equation*}%
{\displaystyle\int\limits_{0}^{1}}
\left(  1-y\right)  ^{p}w_{n}\left(  -y\right)  dy=\frac{1}{p+1}B_{n,p},
\label{37}%
\end{equation*}
where $B_{n,p}$ are $p$-Bernoulli numbers \cite{Rahmani} (see Section 2 for
definitions). The congruence identities of geometric numbers is also one of
the subjects studied. Gross \cite{Gross} showed that
\[
w_{n+4}=w_{n}\text{ }\left(  \operatorname{mod}10\right)  ,
\]
which was generalized by Kauffman \cite{Kaufmann} later. Mez\H{o} \cite{Mezo2}
also gave an elementary proof for Gross' identity. Moreover, Diagana and Ma\"{\i}ga \cite{DM} used $p$-adic Laplace transform and $p$-adic integration
to give some congruences for geometric numbers. We refer the papers
\cite{B3, B1, B2, Dil2, Kargin1, MT} and the references therein more on geometric numbers and polynomials.

In the literature, there are numerous studies for the generalization of geometric
polynomials (e.g. \cite{Dil3, Dil1, Kargin3, Kargin4}). One of the natural extension
of geometric polynomials is the higher order geometric polynomials \cite{B}
\begin{equation}
w_{n}^{\left(  r\right)  }\left(  y\right)  =\sum_{k=0}^{n}%
\genfrac{\{}{\}}{0pt}{}{n}{k}%
\left(  r\right)  _{k}y^{k},\text{ }r>0, \label{20}%
\end{equation}
where $\left(  x\right)  _{n}$ is the Pochhammer symbol defined by $\left(
x\right)  _{n}=x\left(  x+1\right)  \cdots\left(  x+n-1\right)  $ with
$\left(  x\right)  _{0}=1.$ It is evident that $w_{n}^{\left(  1\right)
}\left(  y\right)  =w_{n}\left(  y\right)  $. The polynomials $w_{n}^{\left(
r\right)  }\left(  y\right)  $ have the property
\begin{equation*}
\left(  y\frac{d}{dy}\right)  ^{n}\frac{1}{\left(  1-y\right)  ^{r+1}}%
=\sum_{k=0}^{\infty}\binom{k+r}{k}k^{n}y^{k}=\frac{1}{\left(  1-y\right)
^{r+1}}w_{n}^{\left(  r+1\right)  }\left(  \frac{y}{1-y}\right)  , \label{22}%
\end{equation*}
for any $n,r=0,1,2,\ldots$, and may be defined by means of the exponential
generating function (\cite{B})%
\[
\sum_{n=0}^{\infty}w_{n}^{\left(  r\right)  }\left(  y\right)  \frac{t^{n}%
}{n!}=\left(  \frac{1}{1-y\left(  e^{t}-1\right)  }\right)  ^{r}.
\]
On the other hand, the higher order geometric polynomials and exponential (or
single variable Bell) polynomials
\[
\varphi_{n}\left(  y\right)  =\sum_{k=0}^{n}%
\genfrac{\{}{\}}{0pt}{}{n}{k}%
y^{k}%
\]
are connected by
\begin{equation}
w_{n}^{\left(  r\right)  }\left(  y\right)  =\frac{1}{\Gamma\left(  r\right)
}%
{\displaystyle\int\limits_{0}^{\infty}}
\lambda^{r-1}\varphi_{n}\left(  y\lambda\right)  e^{-\lambda}d\lambda\label{8}%
\end{equation}
(c.f. \cite{B}). According to this integral representation, several generating
functions and recurrence relations for higher order geometric polynomials were
obtained in \cite{BoyadzhievandDil}. Namely, $w_{n+m}^{\left(  r\right)
}\left(  y\right)  $ admits a recurrence relation according to the family
$\left\{  y^{j}w_{n}^{\left(  r+j\right)  }\left(  y\right)  \right\}  $ as
follows:%
\begin{equation}
w_{n+m}^{\left(  r\right)  }\left(  y\right)  =\sum_{k=0}^{n}\sum_{j=0}%
^{m}\binom{n}{k}%
\genfrac{\{}{\}}{0pt}{}{m}{j}%
\left(  r\right)  _{j}j^{n-k}y^{j}w_{k}^{\left(  r+j\right)  }\left(
y\right)  . \label{23}%
\end{equation}

Setting $y=1$ in (\ref{20}), we have higher order geometric numbers
$w_{n}^{\left(  r\right)  }.$ The higher order geometric numbers and geometric
numbers are connected with $w_{n}^{\left(  1\right)  }=w_{n}$ and the formula
\begin{equation}
w_{n}^{\left(  r\right)  }=\frac{1}{r!2^{r}}\sum_{k=0}^{r}%
\genfrac{[}{]}{0pt}{}{r+1}{k+1}%
w_{n+k}, \label{29}%
\end{equation}
which was proved by a combinatorial method in \cite[Theorem 2]{AHLBACH}. Here,
$%
\genfrac{[}{]}{0pt}{}{n}{k}%
$ is the Stirling number of the first kind (\cite{Graham}). Moreover, some
congruence identities for the higher order geometric numbers can also be found
in the recent work \cite{DM}.

In this paper, dealing with two-variable geometric polynomials defined in
\cite{Kim} by%
\begin{equation}
\sum_{n=0}^{\infty}w_{n}^{\left(  r\right)  }\left(  x;y\right)  \frac{t^{n}%
}{n!}=\left(  \frac{1}{1-y\left(  e^{t}-1\right)  }\right)  ^{r}e^{xt},
\label{1}%
\end{equation}
we obtain new recurrence relations, an explicit formula, and a result
generalizing (c.f. \cite{BoyadzhievandDil})%
\begin{equation*}
\sum_{k=0}^{n}\binom{n}{k}w_{k}^{\left(  r\right)  }\left(  y\right)
w_{n-k}\left(  y\right)  =\frac{w_{n+1}^{\left(  r\right)  }\left(  y\right)
+rw_{n}^{\left(  r\right)  }\left(  y\right)  }{r\left(  1+y\right)  },
\label{21}%
\end{equation*}
for higher order geometric polynomials. We particularly use the explicit
formula to obtain an integral representation similar to (\ref{8}) involving
$r$-Bell polynomials, which are defined in \cite{Mezo1} as
\begin{equation}
\varphi_{n,r}\left(  y\right)  =\sum_{k=0}^{n}%
\genfrac{\{}{\}}{0pt}{}{n+r}{k+r}%
_{r}y^{k}, \label{35}%
\end{equation}
where $%
\genfrac{\{}{\}}{0pt}{}{n+r}{k+r}%
_{r}$ are $r$-Stirling numbers of the second kind (\cite{Broder}). The
resulting integral representation enables us to utilize some properties of
$r$-Bell polynomials for higher order geometric polynomials. In particular, we
evaluate the infinite sum
\[
\sum_{k=0}^{\infty}\left(  k+r\right)  ^{n}\binom{k+r-1}{k}x^{k}%
\]
in terms of higher order geometric polynomials, obtain an ordinary generating
function for higher order geometric polynomials, introduce a new recurrence
for $w_{n+m}^{\left(  r\right)  }\left(  y\right)  $, and generalize
(\ref{29}). Moreover, we give an integral representation relating the higher order geometric polynomials and $p$-Bernoulli
numbers, and express properties of $p$-Bernoulli numbers originated from those for the higher order geometric polynomials. We state and prove these results in Section 3,
followed by Section 2, in which we summarize known results that we need in
this paper. Section 4 is the last section of this paper where we prove
congruences for higher order geometric polynomials and $p$-Bernoulli numbers
using some of the results presented in Section 3. Particularly, we state a von
Staudt-Clausen type congruence for $p$-Bernoulli numbers.

\section{Preliminaries}

\setcounter{theorem}{0} \setcounter{equation}{0}

The Stirling numbers of the first kind $%
\genfrac{[}{]}{0pt}{}{n}{k}%
$ can be defined by means of%
\[
x\left(  x+1\right)  \cdots\left(  x+n-1\right)  =\sum_{k=0}^{n}%
\genfrac{[}{]}{0pt}{}{n}{k}%
x^{k}%
\]
or by the generating function%
\[
\left(  -\log\left(  1-x\right)  \right)  ^{k}=k!\sum_{n=k}^{\infty}%
\genfrac{[}{]}{0pt}{}{n}{k}%
\frac{x^{k}}{k!}%
\]
(c.f. \cite{Comtet-1974, Graham}). It follows from either of these
definitions that%
\begin{equation}%
\genfrac{[}{]}{0pt}{}{n}{k}%
=\left(  n-1\right)
\genfrac{[}{]}{0pt}{}{n-1}{k}%
+%
\genfrac{[}{]}{0pt}{}{n-1}{k-1}
\label{31}%
\end{equation}
with%
\[%
\genfrac{[}{]}{0pt}{}{n}{0}%
=0\text{ if }n>0\text{ and }%
\genfrac{[}{]}{0pt}{}{n}{k}%
=0\text{ if }k>n\text{ or }k<0.
\]
We note the following special values which will be used in the sequel:
\begin{align*}%
\genfrac{[}{]}{0pt}{}{n}{1}%
&  =1\text{, }%
\genfrac{[}{]}{0pt}{}{n}{1}%
=\left(  n-1\right)  !\text{ if }n>0\text{, }\\%
\genfrac{[}{]}{0pt}{}{n}{n-1}%
&  =\binom{n}{2}\text{, }%
\genfrac{[}{]}{0pt}{}{n}{n-2}%
=\frac{3n-1}{4}\binom{n}{3}\text{, and }%
\genfrac{[}{]}{0pt}{}{n}{n-3}%
=\binom{n}{2}\binom{n}{4}.
\end{align*}

Many properties of $%
\genfrac{[}{]}{0pt}{}{n}{k}%
$ can be found in \cite[pp. 214-219]{Comtet-1974}. In particular, we have%
\[
k%
\genfrac{[}{]}{0pt}{}{n}{k}%
=\sum_{i=k-1}^{n-1}\binom{n}{i}\left(  n-i-1\right)  !%
\genfrac{[}{]}{0pt}{}{i}{k-1}%
.
\]
This equality can be used to obtain some congruences for $%
\genfrac{[}{]}{0pt}{}{n}{k}%
$. For example, if we take $n=q$, where $q$ is a prime number, then%
\begin{equation}%
\genfrac{[}{]}{0pt}{}{q}{k}%
\equiv0\text{ }\left(  \operatorname{mod}q\right)  \text{ \ \ }\left(
k=2,3,\ldots,q-1\right)  ,\label{32}%
\end{equation}
since%
\[
\binom{q}{i}\equiv0\text{ }\left(  \operatorname{mod}q\right)  \text{
\ \ }\left(  i=1,2,\ldots,q-1\right)  .
\]

The Stirling numbers of the second kind $%
\genfrac{\{}{\}}{0pt}{}{n}{k}%
$ can be defined by means of%
\[
x^{n}=\sum_{k=0}^{n}%
\genfrac{\{}{\}}{0pt}{}{n}{k}%
x\left(  x-1\right)  \cdots\left(  x-k+1\right)  ,
\]
or by the generating function%
\[
\left(  \mathrm{e}^{x}-1\right)  ^{k}=k!\sum_{n=k}^{\infty}%
\genfrac{\{}{\}}{0pt}{}{n}{k}%
\frac{x^{n}}{n!}%
\]
(c.f. \cite{Comtet-1974, Graham}). It follows from the generating function
that%
\[%
\genfrac{\{}{\}}{0pt}{}{n}{k}%
=%
\genfrac{\{}{\}}{0pt}{}{n-1}{k-1}%
+k%
\genfrac{\{}{\}}{0pt}{}{n-1}{k}%
,
\]
with%
\begin{align*}%
\genfrac{\{}{\}}{0pt}{}{n}{0}%
&  =0\text{ if }n>0\text{, \ \ }%
\genfrac{\{}{\}}{0pt}{}{n}{k}%
=0\text{ if }k>n\text{ or }k<0\\%
\genfrac{\{}{\}}{0pt}{}{n}{n}%
&  =1\text{, \ }%
\genfrac{\{}{\}}{0pt}{}{n}{1}%
=1\text{ if }n>0.
\end{align*}
We note the following well-known identity for $%
\genfrac{\{}{\}}{0pt}{}{n}{k}%
$ for future reference:%
\begin{equation}%
\genfrac{\{}{\}}{0pt}{}{n}{k}%
=\frac{1}{k!}\sum_{j=1}^{k}\left(  -1\right)  ^{k-j}\binom{k}{j}%
j^{n}.\label{36}%
\end{equation}

Performing the product of two generating functions for $%
\genfrac{\{}{\}}{0pt}{}{n}{k}%
$, we obtain the convolution formula (\cite{Hsu-Shiue-1998})%
\[
\binom{k_{1}+k_{2}}{k_{1}}%
\genfrac{\{}{\}}{0pt}{}{n}{k_{1}+k_{2}}%
=\sum_{m=0}^{n}\binom{n}{m}%
\genfrac{\{}{\}}{0pt}{}{m}{k_{1}}%
\genfrac{\{}{\}}{0pt}{}{n-m}{k_{2}}%
.
\]
Letting $k=k_{1}+k_{2}$ and $n=q$, a prime number, we deduce that%
\begin{equation}%
\genfrac{\{}{\}}{0pt}{}{q}{k}%
\equiv0\text{ }\left(  \operatorname{mod}q\right)  \text{ \ \ }\left(
k=2,3,\ldots,q-1\right)  ,\label{33}%
\end{equation}
since again%
\[
\binom{q}{i}\equiv0\text{ }\left(  \operatorname{mod}q\right)  \text{
\ \ }\left(  i=1,2,\ldots,q-1\right)  ,
\]
and $1<k<q$.

Stirling numbers have been generalized in many ways. One of them is called
$r$-Stirling numbers (or weighted Stirling numbers). $r$-Stirling numbers of
the second kind $%
\genfrac{\{}{\}}{0pt}{}{n}{k}%
_{r}$ can be defined by means of the generating function (see \cite{Broder})%
\begin{equation}
\left(  \mathrm{e}^{x}-1\right)  ^{k}\mathrm{e}^{rx}=k!\sum_{n=k}^{\infty}%
\genfrac{\{}{\}}{0pt}{}{n}{k}%
_{r}\frac{x^{n}}{n!}. \label{34}%
\end{equation}

The Bernoulli numbers $B_{n}$ are defined by the generating function%
\[
\frac{t}{\mathrm{e}^{t}-1}=\sum_{n=0}^{\infty}B_{n}\frac{t^{n}}{n!}%
\]
or by the equivalent recursion%
\[
B_{0}=1\text{ \ \ and \ \ }\sum_{k=0}^{n-1}\frac{B_{k}}{k!\left(  n-k\right)
!}=0\text{ for }n\geq2.
\]
The first values are%
\[
B_{1}=-\frac{1}{2},B_{2}=\frac{1}{6},B_{4}=-\frac{1}{30},B_{6}=\frac{1}{42},
\]
and $B_{2k+1}=0$ for $k\geq1$. The denominators of the Bernoulli numbers can
be completely determined due to von Staudt-Clausen theorem: for any integer
$n\geq1$, $B_{2n}$ can be written as%
\[
B_{2n}=A_{2n}-\sum_{q:\left(  q-1\right)  |2n}\frac{1}{q},
\]
where $A_{2n}$ is an integer and the sum runs over all the prime numbers such
that $\left(  q-1\right)  |2n$. It can be stated equivalently as%
\[
qB_{2n}\equiv%
\genfrac{\{}{.}{0pt}{}{0\text{ }\left(  \operatorname{mod}q\right)  ,\text{ if
}\left(  q-1\right)  \nmid2n,}{-1\text{ }\left(  \operatorname{mod}q\right)
,\text{ if }\left(  q-1\right)  \mid2n.}%
\]
We note that this classification is also valid for $B_{1}$.

Many generalizations of Bernoulli numbers appear in the literature. One
generalization is the $p$-Bernoulli numbers $B_{n,p}$, which are due to
Rahmani \cite{Rahmani}, defined by means of the generating function%
\[
\sum_{n=0}^{\infty}B_{n,p}\frac{x^{n}}{n!}=\text{ }_{2}F_{1}\left(
1,1;p+2,1-\mathrm{e}^{t}\right)  ,
\]
where $_{2}F_{1}\left(  a,b;c;z\right)  $ is the Gaussian hypergeometric
function%
\[
_{2}F_{1}\left(  a,b;c;z\right)  =\sum_{k=0}^{\infty}\frac{\left(  a\right)
_{k}\left(  b\right)  _{k}}{\left(  c\right)  _{k}}\frac{z^{k}}{k!}.
\]
$p$-Bernoulli numbers are related to Bernoulli numbers in that $B_{n,0}=B_{n}$
and%
\begin{equation}
\sum_{k=0}^{p}%
\genfrac{[}{]}{0pt}{}{p}{k}%
\left(  -1\right)  ^{k}B_{n+k}=\frac{p!}{p+1}B_{n,p},\text{ for }n,p\geq0,
\label{17}%
\end{equation}
and satisfy an explicit formula of the form%
\begin{equation}
B_{n,p}=\frac{p+1}{p!}\sum_{k=0}^{n}%
\genfrac{\{}{\}}{0pt}{}{n+p}{k+p}%
_{p}\frac{\left(  -1\right)  ^{k}\left(  k+p\right)  !}{k+p+1}. \label{12}%
\end{equation}

\section{Recurrence Relations}

\setcounter{theorem}{0} \setcounter{equation}{0}

From the generating function for higher order two-variable geometric
polynomials (\ref{1}) we have
\begin{equation}
w_{n}^{\left(  r\right)  }\left(  x;y\right)  =\sum_{k=0}^{n}\binom{n}{k}%
w_{k}^{\left(  r\right)  }\left(  y\right)  x^{n-k}. \label{14}%
\end{equation}
Then, it is obvious that
\begin{align*}
w_{n}^{\left(  r\right)  }\left(  0;y\right)   &  =w_{n}^{\left(  r\right)
}\left(  y\right)  ,\text{ }w_{n}^{\left(  1\right)  }\left(  x;y\right)
=w_{n}\left(  x;y\right)  ,\text{ }\\
w_{n}^{\left(  r\right)  }\left(  0;1\right)   &  =w_{n}^{\left(  r\right)
}\text{ and }w_{n}^{\left(  1\right)  }\left(  0;1\right)  =w_{n}.
\end{align*}

Setting $x+r$ instead of $x$ in (\ref{1}), we have
\[
w_{n}^{\left(  r\right)  }\left(  x+r;y\right)  =\left(  -1\right)  ^{n}%
w_{n}^{\left(  r\right)  }\left(  -x;-y-1\right)  ,\text{ for }n\geq0.
\]
Then, for $x=0,$ we conclude that
\begin{equation}
w_{n}^{\left(  r\right)  }\left(  r;y\right)  =\left(  -1\right)  ^{n}%
w_{n}^{\left(  r\right)  }\left(  -y-1\right)  , \label{2}%
\end{equation}
a relationship between two-variable and single variable higher order geometric polynomials.

\begin{proposition}
For $n\geq0$ and $r>0,$ we have the following recurrence formulas:%
\begin{equation}
\sum_{k=0}^{n}\binom{n}{k}w_{k}^{\left(  r\right)  }\left(  y\right)
r^{n-k}=\left(  -1\right)  ^{n}w_{n}^{\left(  r\right)  }\left(  -y-1\right)
\label{15}%
\end{equation}
and%
\begin{equation}
\sum_{k=0}^{n}\binom{n}{k}w_{k}^{\left(  r+1\right)  }\left(  y\right)
=\frac{1}{ry}w_{n+1}^{\left(  r\right)  }\left(  y\right)  . \label{16}%
\end{equation}

\end{proposition}

\begin{proof}
Combining (\ref{14}) and (\ref{2}), we have (\ref{15}). Furthermore, taking
$x=1$ in (\ref{14}) and using the recurrence relation presented in
\cite[Theorem 3.4]{Kim}%
\[
w_{n+1}^{\left(  r\right)  }\left(  x;y\right)  =w_{n}^{\left(  r\right)
}\left(  x;y\right)  +ryw_{n}^{\left(  r+1\right)  }\left(  x+1;y\right)
\]
for $x=1,$ we obtain (\ref{16}).
\end{proof}

\begin{theorem}
For every $n\geq0$ and every $r_{1},r_{2}>0$, we have the convolution identity%
\begin{equation}
\sum_{k=0}^{n}\binom{n}{k}w_{k}^{\left(  r_{1}\right)  }\left(  y\right)
w_{n-k}^{\left(  r_{2}\right)  }\left(  y\right)  =\frac{w_{n+1}^{\left(
r_{1}+r_{2}-1\right)  }\left(  y\right)  +\left(  r_{1}+r_{2}-1\right)
w_{n}^{\left(  r_{1}+r_{2}-1\right)  }\left(  y\right)  }{\left(  r_{1}%
+r_{2}-1\right)  \left(  1+y\right)  }. \label{11}%
\end{equation}

\end{theorem}

\begin{proof}
We first note that%
\[
yr\frac{e^{\left(  x+1\right)  t}}{\left(  1-y\left(  e^{t}-1\right)  \right)
^{r+1}}=\frac{d}{dt}\left(  \left(  \frac{1}{1-y\left(  e^{t}-1\right)
}\right)  ^{r}e^{xt}\right)  -\frac{xe^{xt}}{\left(  1-y\left(  e^{t}%
-1\right)  \right)  ^{r}}.
\]
Let $x=x_{1}+x_{2}-1$ and $r=r_{1}+r_{2}-1$. Then by (\ref{1}), product of two
infinite series, and formal differentiation under summation, we obtain that%
\begin{align*}
yr\frac{e^{\left(  x+1\right)  t}}{\left(  1-y\left(  e^{t}-1\right)  \right)
^{r+1}}  &  =y\left(  r_{1}+r_{2}-1\right)  \frac{e^{\left(  x_{1}%
+x_{2}\right)  t}}{\left(  1-y\left(  e^{t}-1\right)  \right)  ^{r_{1}+r_{2}}%
}\label{13}\\
&  =y\left(  r_{1}+r_{2}-1\right)  \sum_{n=0}^{\infty}w_{n}^{\left(
r_{1}\right)  }\left(  x_{1};y\right)  \frac{t^{n}}{n!}\sum_{n=0}^{\infty
}w_{n}^{\left(  r_{2}\right)  }\left(  x_{2};y\right)  \frac{t^{n}}%
{n!}\nonumber\\
&  =y\left(  r_{1}+r_{2}-1\right)  \sum_{n=0}^{\infty}\left[  \sum_{k=0}%
^{n}\binom{n}{k}w_{k}^{\left(  r_{1}\right)  }\left(  x_{1};y\right)
w_{n-k}^{\left(  r_{2}\right)  }\left(  x_{2};y\right)  \right]  \frac{t^{n}%
}{n!},\nonumber
\end{align*}%
\[
\frac{xe^{xt}}{\left(  1-y\left(  e^{t}-1\right)  \right)  ^{r}}=\left(
x_{1}+x_{2}-1\right)  \sum_{n=0}^{\infty}w_{n}^{\left(  r_{1}+r_{2}-1\right)
}\left(  x_{1}+x_{2}-1;y\right)  \frac{t^{n}}{n!},
\]
and
\[
\frac{d}{dt}\left(  \left(  \frac{1}{1-y\left(  e^{t}-1\right)  }\right)
^{r}e^{xt}\right)  =\sum_{n=0}^{\infty}w_{n+1}^{\left(  r_{1}+r_{2}-1\right)
}\left(  x_{1}+x_{2}-1;y\right)  \frac{t^{n}}{n!}.
\]
Equating coefficients of $\frac{t^{n}}{n!}$ on both sides, we derive that%
\begin{align*}
&  \sum_{k=0}^{n}\binom{n}{k}w_{k}^{\left(  r_{1}\right)  }\left(
x_{1};y\right)  w_{n-k}^{\left(  r_{2}\right)  }\left(  x_{2};y\right) \\
&  \quad=\frac{1}{y\left(  r_{1}+r_{2}-1\right)  }\left[  w_{n+1}^{\left(
r_{1}+r_{2}-1\right)  }\left(  x_{1}+x_{2}-1;y\right)  -\left(  x_{1}%
+x_{2}-1\right)  w_{n}^{\left(  r_{1}+r_{2}-1\right)  }\left(  x_{1}%
+x_{2}-1;y\right)  \right]  .
\end{align*}
Setting $x_{1}=r_{1}$, $x_{2}=r_{2}$ and using (\ref{2}) we obtain the
convolution formula (\ref{11}).
\end{proof}

In the following theorem we give a new explicit expression for higher order
geometric polynomials and numbers

\begin{theorem}
\label{teo1}For $n\geq0,$
\begin{equation}
w_{n}^{\left(  r\right)  }\left(  y\right)  =\sum_{k=0}^{n}%
\genfrac{\{}{\}}{0pt}{}{n+r}{k+r}%
_{r}\left(  r\right)  _{k}\left(  -1\right)  ^{n+k}\left(  y+1\right)  ^{k}.
\label{3}%
\end{equation}
In particular%
\[
w_{n}^{\left(  r\right)  }=\sum_{k=0}^{n}%
\genfrac{\{}{\}}{0pt}{}{n+r}{k+r}%
_{r}\left(  r\right)  _{k}\left(  -1\right)  ^{n+k}2^{k}.
\]

\end{theorem}

\begin{proof}
Writing $x=r$ in (\ref{1}), employing the generalized binomial formula, and
using the generating function of $r$-Stirling numbers (\ref{34}) we have%
\begin{align*}
\sum_{n=0}^{\infty}w_{n}^{\left(  r\right)  }\left(  r;y\right)  \frac{t^{n}%
}{n!}  &  =\left(  \frac{1}{1-y\left(  e^{t}-1\right)  }\right)  ^{r}e^{rt}\\
&  =\sum_{k=0}^{\infty}\left(  r\right)  _{k}y^{k}\frac{\left(  e^{t}%
-1\right)  ^{k}}{k!}e^{rt}\\
&  =\sum_{k=0}^{\infty}\sum_{n=k}^{\infty}%
\genfrac{\{}{\}}{0pt}{}{n+r}{k+r}%
_{r}\left(  r\right)  _{k}y^{k}\frac{t^{n}}{n!}\\
&  =\sum_{n=0}^{\infty}\left[  \sum_{k=0}^{n}%
\genfrac{\{}{\}}{0pt}{}{n+r}{k+r}%
_{r}\left(  r\right)  _{k}y^{k}\right]  \frac{t^{n}}{n!}.
\end{align*}
Comparing the coefficients of $\frac{t^{n}}{n!}$ we obtain
\[
w_{n}^{\left(  r\right)  }\left(  r;y\right)  =\sum_{k=0}^{n}%
\genfrac{\{}{\}}{0pt}{}{n+r}{k+r}%
_{r}\left(  r\right)  _{k}y^{k}.
\]
Using (\ref{2}) and replacing $y$ with $-\left(  y+1\right)  ,$ we reach the
desired equation.
\end{proof}

Now, with use of Theorem \ref{teo1}, we connect higher order geometric
polynomials and $r$-Bell polynomials in the following lemma which will be
useful for the subsequent results.

\begin{lemma}
\label{lem1}For every $n\geq0$ and every $r>0,$ we have the integral
representation%
\begin{equation}
\left(  -1\right)  ^{n}w_{n}^{\left(  r\right)  }\left(  -y-1\right)
=\frac{1}{\Gamma\left(  r\right)  }%
{\displaystyle\int\limits_{0}^{\infty}}
\lambda^{r-1}\varphi_{n,r}\left(  y\lambda\right)  e^{-\lambda}d\lambda.
\label{4}%
\end{equation}

\end{lemma}

\begin{proof}
By (\ref{35}) we have%
\begin{align*}%
{\displaystyle\int\limits_{0}^{\infty}}
\lambda^{r-1}\varphi_{n,r}\left(  y\lambda\right)  e^{-\lambda}d\lambda &
=\sum_{k=0}^{n}%
\genfrac{\{}{\}}{0pt}{}{n+r}{k+r}%
_{r}y^{k}%
{\displaystyle\int\limits_{0}^{\infty}}
\lambda^{r+k-1}e^{-\lambda}d\lambda\\
&  =\sum_{k=0}^{n}%
\genfrac{\{}{\}}{0pt}{}{n+r}{k+r}%
_{r}\Gamma\left(  r+k\right)  y^{k}.
\end{align*}
Using (\ref{3}) in the above yields the desired equation.
\end{proof}

Higher order geometric polynomials are seen in the evaluation of the infinite
series
\[
\sum_{k=0}^{\infty}k^{n}\binom{k+r-1}{k}x^{k}=\frac{1}{\left(  1-x\right)
^{r}}w_{n}^{\left(  r\right)  }\left(  \frac{x}{1-x}\right)
\]
If we apply Lemma \ref{lem1} to the Dobinski's formula for $r$-Bell
polynomials
\[
\varphi_{n,r}\left(  y\right)  =\frac{1}{e^{y}}\sum_{n=0}^{\infty}%
\frac{\left(  k+r\right)  ^{n}}{k!}x^{k},
\]
we can evaluate a new infinite series in terms of higher order geometric polynomials.

\begin{theorem}
For every $n\geq0$ and every $r>0,\left\vert y\right\vert <1$%
\[
\sum_{k=0}^{\infty}\left(  k+r\right)  ^{n}\binom{k+r-1}{k}y^{k}=\frac{\left(
-1\right)  ^{n}}{\left(  1-y\right)  ^{r}}w_{n}^{\left(  r\right)  }\left(
\frac{1}{y-1}\right)  .
\]

\end{theorem}

Next we introduce ordinary generating function for higher order geometric polynomials.

\begin{theorem}
For real $y<-\frac{1}{2}$, the higher order
geometric polynomials have the generating function%
\[
\sum_{n=0}^{\infty}w_{n}^{\left(  r\right)  }\left(  y\right)  t^{n}%
=\frac{\left(  -1\right)  ^{r}}{\left(  1+rt\right)  y^{r}}\text{ }_{2}%
F_{1}\left(  \frac{rt+1}{t},r;\frac{rt+t+1}{t};\frac{y+1}{y}\right)  .
\]

\end{theorem}

\begin{proof}
We start by observing the ordinary generating function for $r$-Bell
polynomials (\cite[Theorem 3.2]{Mezo1})%
\[
\sum_{n=0}^{\infty}\varphi_{n,r}\left(  y\right)  t^{n}=\frac{-1}{rt-1}%
\frac{1}{e^{y}}\text{ }_{1}F_{1}\left(  \frac{rt-1}{t};\frac{rt+t-1}%
{t};y\right)  .
\]
In light of the equation (\ref{4}), this equation can be written as%
\begin{align*}
\sum_{n=0}^{\infty}\left(  -1\right)  ^{n}w_{n}^{\left(  r\right)  }\left(
-y-1\right)  t^{n}  &  =\frac{1}{\left(  1-rt\right)  \Gamma\left(  r\right)
}%
{\displaystyle\int\limits_{0}^{\infty}}
\lambda^{r-1}e^{-\left(  y+1\right)  \lambda}\text{ }_{1}F_{1}\left(
\frac{rt-1}{t};\frac{rt+t-1}{t};y\lambda\right)  d\lambda\\
&  =\frac{1}{\left(  1-rt\right)  \Gamma\left(  r\right)  }\sum_{k=0}^{\infty
}\frac{\left(  \frac{rt-1}{t}\right)  _{k}}{\left(  \frac{rt+t-1}{t}\right)
_{k}}\frac{y^{k}}{k!}%
{\displaystyle\int\limits_{0}^{\infty}}
\lambda^{r+k-1}e^{-\left(  y+1\right)  \lambda}d\lambda\\
&  =\frac{1}{\left(  1-rt\right)  \left(  1+y\right)  ^{r}}\sum_{k=0}^{\infty
}\frac{\left(  \frac{rt-1}{t}\right)  _{k}\left(  r\right)  _{k}}{\left(
\frac{rt+t-1}{t}\right)  _{k}k!}\left(  \frac{y}{1+y}\right)  ^{k}\\
&  =\frac{1}{\left(  1-rt\right)  \left(  1+y\right)  ^{r}}\text{ }_{2}%
F_{1}\left(  \frac{rt-1}{t},r;\frac{rt+t-1}{t};\frac{y}{1+y}\right)  .
\end{align*}
We then replace $-\left(  y+1\right)  $ with $y$ and $-t$ with $t$ to obtain
the desired equation.
\end{proof}

Now, we give an alternative representation for $w_{n+m}^{\left(  r\right)
}\left(  y\right)  $, which also generalizes (\ref{29}) in the following theorem.

\begin{theorem}
For all nonnegative integers $n$, $m$, $r$ and $p$, we have%
\begin{equation}
w_{n+m}^{\left(  r\right)  }\left(  y\right)  =\sum_{k=0}^{m}%
\genfrac{\{}{\}}{0pt}{}{m+r}{k+r}%
_{r}\left(  r\right)  _{k}\left(  -1\right)  ^{m+k}\left(  y+1\right)
^{k}w_{n}^{\left(  r+k\right)  }\left(  y\right)  \label{5}%
\end{equation}
and%
\begin{equation}
w_{n}^{\left(  r+p\right)  }\left(  y\right)  =\frac{1}{\left(  p\right)
_{r}\left(  1+y\right)  ^{p}}\sum_{k=0}^{p}%
\genfrac{[}{]}{0pt}{}{p+r}{k+r}%
_{r}w_{n+k}^{\left(  r\right)  }\left(  y\right)  . \label{7}%
\end{equation}

\end{theorem}

\begin{proof}
We prove (\ref{5}) first. Using the following property of $r$-Bell polynomials
presented in \cite[Eq. (8)]{Mihoubi}%
\[
\varphi_{n+m,r}\left(  y\right)  =\sum_{k=0}^{m}%
\genfrac{\{}{\}}{0pt}{}{m+r}{k+r}%
_{r}y^{k}\varphi_{n,r+k}\left(  y\right)
\]
in (\ref{4}), we have%
\begin{align*}
\left(  -1\right)  ^{n+m}w_{n+m}^{\left(  r\right)  }\left(  -y-1\right)   &
=\sum_{k=0}^{m}%
\genfrac{\{}{\}}{0pt}{}{m+r}{k+r}%
_{r}y^{k}\frac{\Gamma\left(  k+r\right)  }{\Gamma\left(  r\right)  }\frac
{1}{\Gamma\left(  k+r\right)  }%
{\displaystyle\int\limits_{0}^{\infty}}
\lambda^{r+k-1}\varphi_{n,r+k}\left(  y\lambda\right)  e^{-\lambda}d\lambda\\
&  =\sum_{k=0}^{m}%
\genfrac{\{}{\}}{0pt}{}{m+r}{k+r}%
_{r}\left(  r\right)  _{k}y^{k}\left(  -1\right)  ^{n}w_{n}^{\left(  r\right)
}\left(  -y-1\right)
\end{align*}
which is equal to (\ref{5}).

To prove (\ref{7}), we use the formula
\[
y^{p}\varphi_{n,r+p}\left(  y\right)  =\sum_{k=0}^{p}%
\genfrac{[}{]}{0pt}{}{p+r}{k+r}%
_{r}\left(  -1\right)  ^{p-k}\varphi_{n+k,p}\left(  y\right)
\]
(\cite[Eq. (11)]{Mihoubi}) in (\ref{4}).
\end{proof}

Using (\ref{15}) in (\ref{5}), we obtain the following result similar which is
slightly different from (\ref{23}).

\begin{corollary}%
\[
w_{n+m}^{\left(  r\right)  }\left(  y\right)  =\sum_{k=0}^{n}\sum_{j=0}^{m}%
\genfrac{\{}{\}}{0pt}{}{m+r}{j+r}%
_{r}\binom{n}{k}\left(  j+r\right)  ^{n-k}\left(  -1\right)  ^{n+m+j}\left(
r\right)  _{j}\left(  y+1\right)  ^{j}w_{k}^{\left(  r+j\right)  }\left(
-y-1\right)  .
\]

\end{corollary}

We note that it is also possible to derive this result by applying (\ref{8})
and (\ref{4}) in
\[
\varphi_{n+m,r}\left(  y\right)  =\sum_{k=0}^{n}\sum_{j=0}^{m}\binom{n}{k}%
\genfrac{\{}{\}}{0pt}{}{m+r}{j+r}%
_{r}\left(  j+r\right)  ^{n-k}y^{j}\varphi_{k}\left(  y\right)  ,
\]
a formula given in \cite[Eq. (9)]{Mihoubi}. Moreover, for $r=1$, (\ref{7}) can
be written as
\begin{equation}
\left(  1+y\right)  ^{p}w_{n}^{\left(  p+1\right)  }\left(  y\right)
=\frac{1}{p!}\sum_{k=0}^{p}%
\genfrac{[}{]}{0pt}{}{p+1}{k+1}%
w_{n+k}\left(  y\right)  , \label{9}%
\end{equation}
which is also polynomial extension of (\ref{29}). Replacing $y$ by $-y$ and
integrating both sides with respect to $y$ from $0$ to $1$, we have%
\[%
{\displaystyle\int\limits_{0}^{1}}
\left(  1-y\right)  ^{p}w_{n}^{\left(  p+1\right)  }\left(  -y\right)
dy=\frac{1}{p!}\sum_{k=0}^{p}%
\genfrac{[}{]}{0pt}{}{p+1}{k+1}%
B_{n+k}.
\]
Then using (\ref{17}), we obtain the following integral representation for
$p$-Bernoulli numbers.

\begin{theorem}
\label{teo2}For $n\geq1$ and $p\geq0$%
\begin{equation}%
{\displaystyle\int\limits_{0}^{1}}
\left(  1-y\right)  ^{p}w_{n}^{\left(  p+1\right)  }\left(  -y\right)
dy=\left(  -1\right)  ^{n-1}\frac{p+1}{p+2}B_{n-1,p+1}. \label{10}%
\end{equation}

\end{theorem}

\noindent The explicit formula (\ref{12}) for $p$-Bernoulli numbers can be
also deduced using this integral representation in (\ref{3}).

The following theorem generalizes the identities (\ref{12}) and (\ref{17}).

\begin{theorem}
For $n,p,m\geq0,$ we have
\begin{equation}
B_{n+m,p}=\left(  p+1\right)  \sum_{k=0}^{m}%
\genfrac{\{}{\}}{0pt}{}{m+p}{k+p}%
_{p}\frac{\left(  -1\right)  ^{k}\left(  p+1\right)  _{k}}{k+p+1}B_{n,p+k}.
\label{18}%
\end{equation}
For $n,r\geq1$ and $p\geq0$ we have%
\begin{equation}
B_{n,p+r}=\frac{r\left(  p+r+1\right)  }{\left(  r+1\right)  \left(  p\right)
_{r+1}}\sum_{k=0}^{p}%
\genfrac{[}{]}{0pt}{}{p+r}{k+r}%
_{r}\left(  -1\right)  ^{k}B_{n+k,r}.\text{ } \label{19}%
\end{equation}

\end{theorem}

\begin{proof}
Firstly, we replace $y$ with $-y$ in (\ref{7}), multiply both sides by
$\left(  1-y\right)  ^{r-1}$, and integrate with respect to $y$ from $0$ to
$1$. The result is%
\[%
{\displaystyle\int\limits_{0}^{1}}
\left(  1-y\right)  ^{p+r-1}w_{n}^{\left(  r+p\right)  }\left(  -y\right)
dy=\frac{1}{\left(  p\right)  _{r}}\sum_{k=0}^{p}%
\genfrac{[}{]}{0pt}{}{p+r}{k+r}%
_{r}%
{\displaystyle\int\limits_{0}^{1}}
\left(  1-y\right)  ^{r-1}w_{n+k}^{\left(  r\right)  }\left(  -y\right)  dy.
\]
From (\ref{10}) this equation turns into%
\[
B_{n-1,p+r}=\frac{r\left(  p+r+1\right)  }{\left(  r+1\right)  \left(
p\right)  _{r+1}}\sum_{k=0}^{p}%
\genfrac{[}{]}{0pt}{}{p+r}{k+r}%
_{r}\left(  -1\right)  ^{k}B_{n+k-1,r}.
\]
Replacing $n$ with $n+1$ in the above equation completes the proof (\ref{19}).

Applying the same method to the identity (\ref{5}) gives (\ref{18}).
\end{proof}

\section{Congruences}

\setcounter{theorem}{0} \setcounter{equation}{0}

In this section we first consider congruences modulo a prime number $q$ for
higher order geometric polynomials. We start with two auxiliary results.

\begin{lemma}
\label{lem2} Let $q$ be an odd prime and $y$ be an integer. Then, we have%
\[
w_{q}\left(  y\right)  \equiv y\left(  \operatorname{mod}q\right)  .
\]

\end{lemma}

\begin{proof}
From (\ref{30}), we have%
\[
w_{q}\left(  y\right)  =%
\genfrac{\{}{\}}{0pt}{}{q}{0}%
+%
\genfrac{\{}{\}}{0pt}{}{q}{1}%
y+%
\genfrac{\{}{\}}{0pt}{}{q}{q}%
q!y^{q}+\sum_{k=2}^{q-1}%
\genfrac{\{}{\}}{0pt}{}{q}{k}%
k!y^{k}.
\]
Since%
\[%
\genfrac{\{}{\}}{0pt}{}{q}{0}%
=0,\text{ }%
\genfrac{\{}{\}}{0pt}{}{q}{1}%
=%
\genfrac{\{}{\}}{0pt}{}{q}{q}%
=1\text{ }%
\]
and by (\ref{33})
\[
k!%
\genfrac{\{}{\}}{0pt}{}{q}{k}%
\equiv0\text{ }\left(  \operatorname{mod}q\right)  ,\text{ }k=2,3,\ldots q-1,
\]
we have%
\[
w_{q}\left(  y\right)  \equiv y+q!y^{q}\equiv y\text{ }\left(
\operatorname{mod}q\right)  ,
\]
the desired result.
\end{proof}

\begin{lemma}
\label{lem3} Let $q$ be a prime and $y$ be an integer. Then for all $n\geq1$,
we have
\[
w_{q+n-1}\left(  y\right)  \equiv w_{n}\left(  y\right)  \left(
\operatorname{mod}q\right)  .
\]

\end{lemma}

\begin{proof}
If $q=2$, then by (\ref{30})%
\begin{align*}
w_{n+1}\left(  y\right)  -w_{n}\left(  y\right)   &  =\sum_{k=0}^{n+1}%
\genfrac{\{}{\}}{0pt}{}{n+1}{k}%
k!y^{k}-\sum_{k=0}^{n}%
\genfrac{\{}{\}}{0pt}{}{n}{k}%
k!y^{k}\\
&  =\left(  n+1\right)  !y^{n+1}+\sum_{k=2}^{n}\left(
\genfrac{\{}{\}}{0pt}{}{n+1}{k}%
-%
\genfrac{\{}{\}}{0pt}{}{n}{k}%
\right)  k!y^{k}\\
&  \equiv0\text{ }\left(  \operatorname{mod}2\right)
\end{align*}
since $%
\genfrac{\{}{\}}{0pt}{}{n}{0}%
=0$ and $%
\genfrac{\{}{\}}{0pt}{}{n}{1}%
=1$ for $n>0$.

Now, suppose that $q$ is an odd prime and let $n\geq q-1$. Then again by
(\ref{30}) we write%
\begin{align*}
w_{q+n-1}\left(  y\right)  -w_{n}\left(  y\right)   &  =\sum_{k=0}^{q+n-1}%
\genfrac{\{}{\}}{0pt}{}{q+n-1}{k}%
k!y^{k}-\sum_{k=0}^{n}%
\genfrac{\{}{\}}{0pt}{}{n}{k}%
k!y^{k}\\
&  =\sum_{k=0}^{q-1}\left(
\genfrac{\{}{\}}{0pt}{}{q+n-1}{k}%
-%
\genfrac{\{}{\}}{0pt}{}{n}{k}%
\right)  k!y^{k}\\
&  +\sum_{k=q}^{q+n-1}%
\genfrac{\{}{\}}{0pt}{}{q+n-1}{k}%
k!y^{k}-\sum_{k=q}^{n}%
\genfrac{\{}{\}}{0pt}{}{n}{k}%
k!y^{k}\\
&  \equiv\sum_{k=0}^{q-1}\left(
\genfrac{\{}{\}}{0pt}{}{q+n-1}{k}%
-%
\genfrac{\{}{\}}{0pt}{}{n}{k}%
\right)  k!y^{k}\text{ }\left(  \operatorname{mod}q\right)  .
\end{align*}
Using (\ref{36}) we obtain that%
\[
w_{q+n-1}\left(  y\right)  -w_{n}\left(  y\right)  \equiv\sum_{k=0}%
^{q-1}k!y^{k}\frac{1}{k!}\sum_{j=1}^{k}\left(  -1\right)  ^{k-j}\binom{k}%
{j}j^{n}\left(  j^{q-1}-1\right)  \equiv0\text{ }\left(  \operatorname{mod}%
q\right)
\]
since $\left(  j,q\right)  =1$ and $j^{q-1}-1\equiv0$ $\left(
\operatorname{mod}q\right)  $.

If $1\leq n<q-1$, then we write%
\begin{align*}
w_{q+n-1}\left(  y\right)  -w_{n}\left(  y\right)   &  =\sum_{k=0}^{q+n-1}%
\genfrac{\{}{\}}{0pt}{}{q+n-1}{k}%
k!y^{k}-\sum_{k=0}^{n}%
\genfrac{\{}{\}}{0pt}{}{n}{k}%
k!y^{k}\\
&  =\sum_{k=0}^{q-1}\left(
\genfrac{\{}{\}}{0pt}{}{q+n-1}{k}%
-%
\genfrac{\{}{\}}{0pt}{}{n}{k}%
\right)  k!y^{k}\\
&  +\sum_{k=q}^{q+n-1}%
\genfrac{\{}{\}}{0pt}{}{q+n-1}{k}%
k!y^{k}\\
&  -\sum_{k=n}^{q-1}%
\genfrac{\{}{\}}{0pt}{}{n}{k}%
k!y^{k}+\sum_{k=n+1}^{q-1}%
\genfrac{\{}{\}}{0pt}{}{n}{k}%
k!y^{k}\\
&  \equiv\sum_{k=0}^{q-1}\sum_{j=1}^{k}\left(  -1\right)  ^{k-j}\binom{k}%
{j}y^{k}j^{n}\left(  j^{q-1}-1\right)  \equiv0\text{ }\left(
\operatorname{mod}q\right)
\end{align*}
since $%
\genfrac{\{}{\}}{0pt}{}{n}{k}%
=0$ when $k>n$.

Therefore, for $n\geq1$, $w_{q+n-1}\left(  y\right)  \equiv w_{n}\left(
y\right)  $ $\left(  \operatorname{mod}q\right)  $.
\end{proof}

We note that a more general result can be found in \cite{AJ2018} for Fubini numbers.

\begin{theorem}
\label{teo4}Let $q$ be an odd prime. If $1+y$ is not a multiple of $q$, then
$w_{q}^{\left(  q\right)  }\left(  y\right)  \equiv0$ $\left(
\operatorname{mod}q\right)  $.
\end{theorem}

\begin{proof}
We set $p=q-1$ and $n=q$ in (\ref{9}) to obtain%
\begin{align*}
&  \left(  1+y\right)  ^{q-1}\left(  q-1\right)  !w_{q}^{\left(  q\right)
}\left(  y\right) \\
&  \quad=\sum_{k=1}^{q}%
\genfrac{[}{]}{0pt}{}{q}{k}%
w_{q+k-1}\left(  y\right) \\
&  \quad=%
\genfrac{[}{]}{0pt}{}{q}{1}%
w_{q}\left(  y\right)  +%
\genfrac{[}{]}{0pt}{}{q}{q}%
w_{1}\left(  y\right)  +\sum_{k=2}^{q-1}%
\genfrac{[}{]}{0pt}{}{q}{k}%
w_{q+k-1}\left(  y\right) \\
&  \quad=\left(  q-1\right)  !w_{q}\left(  y\right)  +y+\sum_{k=2}^{q-1}%
\genfrac{[}{]}{0pt}{}{q}{k}%
w_{q+k-1}\left(  y\right)  .
\end{align*}
By Lemma \ref{lem2} and Lemma \ref{lem3}, we find that%
\begin{align*}
\left(  1+y\right)  ^{q-1}\left(  q-1\right)  !w_{q}^{\left(  q\right)
}\left(  y\right)   &  \equiv\left(  -1\right)  y+y+\sum_{k=2}^{q-1}%
\genfrac{[}{]}{0pt}{}{q}{k}%
w_{k}\left(  y\right)  \text{ }\left(  \operatorname{mod}q\right) \\
&  \equiv0\text{ }\left(  \operatorname{mod}q\right)  ,
\end{align*}
since by (\ref{32})$%
\genfrac{[}{]}{0pt}{}{q}{k}%
\equiv0$ $\left(  \operatorname{mod}q\right)  $ for $2\leq k\leq q-1$ and
$w_{k}\left(  y\right)  $ is an integer when $y$ is an integer. The result now
follows from Fermat's and Wilson's theorems.
\end{proof}

It is obvious from (\ref{20}) that if $y$ is an integer which is a multiple of
$q$, then $w_{n}^{\left(  r\right)  }\left(  y\right)  \equiv0$ $\left(
\operatorname{mod}q\right)  $, since $%
\genfrac{\{}{\}}{0pt}{}{n}{k}%
\left(  r\right)  _{k}$ is an integer. We note that Theorem \ref{teo4} is a
special case which can be drawn from the following result.

\begin{theorem}
If $y$ is an integer that is not a multiple of $q$, then $w_{n}^{\left(
r\right)  }\left(  y\right)  \equiv0$ $\left(  \operatorname{mod}q\right)  $
for $n\geq1$ and $r\equiv0$ $\left(  \operatorname{mod}q\right)  $.
\end{theorem}

\begin{proof}
Let $r=tq$ for some integer $t$. By (\ref{20}), we have%
\[
w_{n}^{\left(  r\right)  }\left(  y\right)  =\sum_{k=0}^{n}k!\binom{tq+k-1}{k}%
\genfrac{\{}{\}}{0pt}{}{n}{k}%
y^{k}.
\]
Since%
\[
k!\binom{tq+k-1}{k}=\left(  tq+k-1\right)  \left(  tq+k-2\right)
\cdots\left(  tq+1\right)  \left(  tq\right)  \equiv0\text{ }\left(
\operatorname{mod}q\right)  ,
\]
we have the result.
\end{proof}

\begin{theorem}
If $y$ is an integer such that $y$ and $1+y$ are not multiples of an odd prime
$q$, then $w_{q-1}^{\left(  r\right)  }\left(  y\right)  \equiv0$ $\left(
\operatorname{mod}q\right)  $ for $r\equiv1$ $\left(  \operatorname{mod}%
q\right)  $.
\end{theorem}

\begin{proof}
Let $r=1+tq$ for some integer $t$. By (\ref{20}), we have%
\[
w_{q-1}^{\left(  r\right)  }\left(  y\right)  =\sum_{k=0}^{q-1}k!\binom
{tq+k}{k}%
\genfrac{\{}{\}}{0pt}{}{q-1}{k}%
y^{k}.
\]
Since%
\[
\binom{tq+k}{k}=\frac{\left(  tq+k\right)  \left(  tq+k-1\right)
\cdots\left(  tq+1\right)  }{k!}\equiv\frac{k\left(  k-1\right)  \cdots1}%
{k!}=1\left(  \operatorname{mod}q\right)  ,
\]
we deduce that%
\[
w_{q-1}^{\left(  r\right)  }\left(  y\right)  \equiv\sum_{k=0}^{q-1}k!%
\genfrac{\{}{\}}{0pt}{}{q-1}{k}%
y^{k}\text{ }\left(  \operatorname{mod}q\right)  .
\]
It follows from (\ref{36}) that%
\[
k!%
\genfrac{\{}{\}}{0pt}{}{q-1}{k}%
\equiv\left(  -1\right)  ^{k-1}\text{ }\left(  \operatorname{mod}q\right)
\]
for $1\leq k\leq q-1$. Since $%
\genfrac{\{}{\}}{0pt}{}{q-1}{0}%
=0$, we then have%
\[
w_{q-1}^{\left(  r\right)  }\left(  y\right)  \equiv\sum_{k=0}^{q-1}\left(
-1\right)  ^{k-1}y^{k}=1-\sum_{k=0}^{q-1}\left(  -1\right)  ^{k}y^{k}%
=1-\frac{1+y^{q}}{1+y}\text{ }\left(  \operatorname{mod}q\right)  ,
\]
which implies%
\[
\left(  1+y\right)  w_{q-1}^{\left(  r\right)  }\left(  y\right)
\equiv1+y-1+y^{q}\equiv0\text{ }\left(  \operatorname{mod}q\right)  ,
\]
and the result.
\end{proof}

These results and their proofs are direct generalizations of the corresponding
congruences for higher order geometric numbers given in \cite[Corollary
4.2]{DM}.

We conclude the study of congruences for higher order geometric polynomials by
a similar result.

\begin{theorem}
If $y$ is an integer that is not a multiple of an odd prime $q$, then
$w_{q+1}^{\left(  r\right)  }\left(  y\right)  \equiv0$ $\left(
\operatorname{mod}q\right)  $ for $r\equiv0$ $\left(  \operatorname{mod}%
q\right)  $, and $w_{q+1}^{\left(  r\right)  }\left(  y\right)  \equiv-y$
$\left(  \operatorname{mod}q\right)  $ for $r\equiv-1$ $\left(
\operatorname{mod}q\right)  $.
\end{theorem}

\begin{proof}
For a prime $q$ and nonnegative integer $m$, we have%
\[%
\genfrac{\{}{\}}{0pt}{}{q+m}{k}%
\equiv%
\genfrac{\{}{\}}{0pt}{}{m+1}{k}%
+%
\genfrac{\{}{\}}{0pt}{}{m}{k-q}%
\text{ }\left(  \operatorname{mod}q\right)  .
\]
This result was given by Howard in \cite{Howard1990}, and can be easily
verified by induction on $m$. It then follows that $%
\genfrac{\{}{\}}{0pt}{}{q+1}{k}%
\equiv0$ $\left(  \operatorname{mod}q\right)  $ for $k=3,4,\ldots,q$, and $%
\genfrac{\{}{\}}{0pt}{}{q+1}{2}%
\equiv1$ $\left(  \operatorname{mod}q\right)  $.

Now, we write (\ref{20}) as%
\begin{align*}
w_{q+1}^{\left(  r\right)  }\left(  y\right)    & =\sum_{k=0}^{q+1}%
k!\binom{r+k-1}{k}%
\genfrac{\{}{\}}{0pt}{}{q+1}{k}%
y^{k}\\
& =ry+r\left(  r+1\right)
\genfrac{\{}{\}}{0pt}{}{q+1}{2}%
y^{2}\\
& +\left(  q+1\right)  !\binom{r+q}{q}y^{q+1}\\
& +\sum_{k=3}^{q}k!\binom{r+k-1}{k}%
\genfrac{\{}{\}}{0pt}{}{q+1}{k}%
y^{k}\\
& \equiv ry+r\left(  r+1\right)  y^{2}\text{ }\left(  \operatorname{mod}%
q\right)  ,
\end{align*}
from which the results follow.
\end{proof}

In the rest of this section we consider congruences for $p$-Bernoulli numbers.
In particular, the following theorem states a von Staudt-Clausen-type result
for $p$-Bernoulli numbers.

\begin{theorem}
For and odd prime $q$ and positive integer $n$, we have%
\[
qB_{2n,q}\equiv%
\genfrac{\{}{.}{0pt}{}{\mp1\text{ }\left(  \operatorname{mod}q\right)  ,\text{
if }\left(  q-1\right)  \nmid2n,}{-\frac{1}{2}\text{ }\left(
\operatorname{mod}q\right)  ,\text{ if }\left(  q-1\right)  \mid2n.}%
\]

\end{theorem}

\begin{proof}
The von Staudt-Clausen theorem can be equivalently stated that%
\[
B_{2n}\equiv%
\genfrac{\{}{.}{0pt}{}{0\text{ }\left(  \operatorname{mod}q\right)  ,\text{ if
}\left(  q-1\right)  \nmid2n,}{\frac{-1}{q}\text{ }\left(  \operatorname{mod}%
q\right)  ,\text{ if }\left(  q-1\right)  \mid2n,}%
\]
where $q$ is a prime and $n\geq1.$ Let $q$ be an odd prime, and replace $n$ by
$2n$ and $p$ by $q$ in (\ref{17}). Then we obtain%
\begin{align*}
\frac{q!}{q+1}B_{2n,q}  &  =\sum_{k=0}^{q}%
\genfrac{[}{]}{0pt}{}{q}{k}%
\left(  -1\right)  ^{k}B_{2n+k}\\
&  =%
\genfrac{[}{]}{0pt}{}{q}{0}%
B_{2n}-%
\genfrac{[}{]}{0pt}{}{q}{1}%
B_{2n+1}+%
\genfrac{[}{]}{0pt}{}{q}{q-1}%
B_{2n+q-1}-%
\genfrac{[}{]}{0pt}{}{q}{q}%
B_{2n+q}\\
&  +\sum_{k=2}^{q-2}%
\genfrac{[}{]}{0pt}{}{q}{k}%
\left(  -1\right)  ^{k}B_{2n+k}.
\end{align*}
Since%
\[%
\genfrac{[}{]}{0pt}{}{q}{0}%
=0,\text{ }%
\genfrac{[}{]}{0pt}{}{q}{1}%
=\left(  q-1\right)  !,\text{ }%
\genfrac{[}{]}{0pt}{}{q}{q-1}%
=\frac{q\left(  q-1\right)  }{2},\text{ }%
\genfrac{[}{]}{0pt}{}{q}{q}%
=1
\]
and $B_{2n+1}=0,$ $n\geq1,$ the above equality turns into
\[
\frac{q!}{q+1}B_{2n,q}=\frac{q\left(  q-1\right)  }{2}B_{2n+q-1}+\sum
_{k=2}^{q-2}%
\genfrac{[}{]}{0pt}{}{q}{k}%
\left(  -1\right)  ^{k}B_{2n+k}.
\]
Now,
if $\left(  q-1\right)  \nmid2n,$ then $\left(  q-1\right)  \nmid2n+q-1,$ so
$qB_{2n+q-1}\equiv0$ $\left(  \operatorname{mod}q\right)  .$ Further, there is
only one $k,$ say $k_{1},$ in $\left\{  2,3,\ldots,q-2\right\}  $ such that
$\left(  q-1\right)  \mid2n+k.$ Indeed, if $\left(  q-1\right)  \nmid2n,$ then
we may write $2n=m\left(  q-1\right)  +a,$ where $m>1$ and $0<a<q-1.$ Since
$1<k<q-1$ and $0<a<q-1,$ we conclude that $1<a+k<2\left(  q-1\right)  ,$ that
is, $a+k$ may be $q-1$ for some values of $a$ and $k.$ Thus, if $\left(
q-1\right)  \nmid2n,$ then
\begin{align*}
\frac{q!}{q+1}B_{2n,q}  &  \equiv%
\genfrac{[}{]}{0pt}{}{q}{k_{1}}%
\left(  -1\right)  ^{k_{1}}B_{2n+k_{1}}+\sum_{\substack{k=2\\\left(
q-1\right)  \nmid2n+k}}^{q-2}%
\genfrac{[}{]}{0pt}{}{q}{k}%
\left(  -1\right)  ^{k}B_{2n+k}\\
&  \equiv\mp1\text{ }\left(  \operatorname{mod}q\right)
\end{align*}
or equivalently%
\[
qB_{2n,q}\equiv\mp1\text{ }\left(  \operatorname{mod}q\right)
\]
by Wilson's theorem. On the other hand, if $\left(  q-1\right)  \mid2n,$ then
$\left(  q-1\right)  \mid2n+q-1,$ so $qB_{2n+q-1}\equiv-1$ $\left(
\operatorname{mod}q\right)  .$ We also have $\left(  q-1\right)  \nmid2n+k$
for $k=2,3,\ldots,q-2,$ so $B_{2n+k}\equiv0$ $\left(  \operatorname{mod}%
q\right)  .$ Thus we obtain%
\[
\frac{q!}{q+1}B_{2n,q}\equiv-\frac{q-1}{2}\text{ }\left(  \operatorname{mod}%
q\right)
\]
or equivalently%
\[
qB_{2n,q}\equiv-\frac{1}{2}\text{ }\left(  \operatorname{mod}q\right)  ,
\]
again by Wilson's theorem.
\end{proof}

In the following theorem, we give a congruence for $B_{q,q},$ where $q>3$ is a prime.

\begin{theorem}
\label{teo3}For a prime $q>3,$ we have%
\[
qB_{q,q}\equiv\frac{1}{12}\text{ }\left(  \operatorname{mod}q\right)  .
\]

\end{theorem}

\begin{proof}
Let $q>3$ be a prime. We write $n=p=q$ in (\ref{17}) to obtain%
\[
\frac{q!}{q+1}B_{q,q}=\sum_{k=0}^{q}%
\genfrac{[}{]}{0pt}{}{q}{k}%
\left(  -1\right)  ^{k}B_{q+k}%
\]
that is,%
\begin{align*}
\frac{q!}{q+1}B_{q,q}  &  =%
\genfrac{[}{]}{0pt}{}{q}{0}%
B_{q}-%
\genfrac{[}{]}{0pt}{}{q}{1}%
B_{q+1}+%
\genfrac{[}{]}{0pt}{}{q}{q-1}%
B_{2q-1}-%
\genfrac{[}{]}{0pt}{}{q}{q}%
B_{2q}\\
&  -%
\genfrac{[}{]}{0pt}{}{q}{q-2}%
B_{2q-2}+\sum_{k=2}^{q-3}%
\genfrac{[}{]}{0pt}{}{q}{k}%
\left(  -1\right)  ^{k}B_{q+k}\\
&  =-\left(  q-1\right)  !B_{q+1}-B_{2q}-%
\genfrac{[}{]}{0pt}{}{q}{q-2}%
B_{2q-2}+\sum_{k=2}^{q-3}%
\genfrac{[}{]}{0pt}{}{q}{k}%
\left(  -1\right)  ^{k}B_{q+k}.
\end{align*}
We note that $\left(  q-1\right)  \nmid\left(  q+k\right)  $ for $k=2,3,\ldots
q-3.$ This is because if we would have $\left(  q-1\right)  \mid\left(
q+k\right)  ,$ then $k=\left(  m-1\right)  q-m$ for some integer $m>2.$ Then
\begin{align*}
2  &  \leq k\leq q-3\Rightarrow2\leq\left(  m-1\right)  q-m\leq q-3\\
&  \Rightarrow mq-q-m\leq q-3\Rightarrow\left(  m-2\right)  q\leq m-3\\
&  \Rightarrow q\leq\frac{m-3}{m-2}=1-\frac{1}{m-2},\text{ }%
\end{align*}
which is impossible since $q>3.$ Hence $B_{q+k}\equiv0$ $\left(
\operatorname{mod}q\right)  $ for $k=2,3,\ldots q-3,$ so the above vanishes
modulo $q$. For the other terms, we note that $B_{q+1}\equiv0$ $\left(
\operatorname{mod}q\right)  ,$ $B_{2q}\equiv0$ $\left(  \operatorname{mod}%
q\right)  $ by von Staudt-Clausen theorem, and that
\[%
\genfrac{[}{]}{0pt}{}{q}{q-2}%
=\frac{\left(  3q-1\right)  q\left(  q-1\right)  \left(  q-2\right)  }{24}.
\]
Thus,
\[
\frac{q!}{q+1}B_{q,q}\equiv-\frac{\left(  3q-1\right)  \left(  q-1\right)
\left(  q-2\right)  }{24}qB_{2\left(  q-1\right)  }\text{ }\left(
\operatorname{mod}q\right)  ,
\]
which reduces to
\[
qB_{q,q}\equiv\frac{1}{12}\text{ }\left(  \operatorname{mod}q\right)
\]
by Wilson's and von Staudt-Clausen theorems.
\end{proof}

Finally, we give a congruence for $B_{q,q+1},$ where $q>3$ is a prime.

\begin{theorem}
For a prime $q>3,$ we have%
\[
B_{q,q+1}\equiv\frac{1}{12}\text{ }\left(  \operatorname{mod}q\right)  .
\]

\end{theorem}

\begin{proof}
Let $q>3$ be a prime. We write $n=p=q$ and $r=1$ in (\ref{19}) to obtain%
\[
B_{q,q+1}=\frac{q+2}{2q\left(  q+1\right)  }\sum_{k=0}^{q}%
\genfrac{[}{]}{0pt}{}{q+1}{k+1}%
\left(  -1\right)  ^{k}B_{q+k,1}.
\]
Since $B_{n,1}=-2B_{n+1},$ we have%
\[
B_{q,q+1}=\frac{q+2}{q\left(  q+1\right)  }\sum_{k=0}^{q}%
\genfrac{[}{]}{0pt}{}{q+1}{k+1}%
\left(  -1\right)  ^{k-1}B_{q+k+1}.
\]
Using (\ref{31}), we obtain that%
\begin{align*}
B_{q,q+1}  &  =\frac{q+2}{q+1}\sum_{k=0}^{q}%
\genfrac{[}{]}{0pt}{}{q}{k+1}%
\left(  -1\right)  ^{k-1}B_{q+k+1}\\
&  +\frac{q+2}{q\left(  q+1\right)  }\sum_{k=0}^{q}%
\genfrac{[}{]}{0pt}{}{q}{k}%
\left(  -1\right)  ^{k-1}B_{q+k+1}\\
&  =\frac{q+2}{q+1}\sum_{k=1}^{q}%
\genfrac{[}{]}{0pt}{}{q}{k}%
\left(  -1\right)  ^{k}B_{q+k}\\
&  +\frac{q+2}{q\left(  q+1\right)  }\sum_{k=0}^{q}%
\genfrac{[}{]}{0pt}{}{q}{k}%
\left(  -1\right)  ^{k-1}B_{q+k+1}.
\end{align*}
Now, using (\ref{17}), we have%
\begin{align*}
\frac{q+2}{q+1}\sum_{k=1}^{q}%
\genfrac{[}{]}{0pt}{}{q}{k}%
\left(  -1\right)  ^{k}B_{q+k}  &  =\frac{q+2\left(  q-1\right)  !}{\left(
q+1\right)  ^{2}}qB_{q,q}\\
&  \equiv-\frac{1}{6}\text{ }\left(  \operatorname{mod}q\right)
\end{align*}
by Theorem \ref{teo3} and Wilson's theorem. Therefore, we consider the sum%
\begin{align*}
&  \frac{q+2}{q\left(  q+1\right)  }\sum_{k=0}^{q}%
\genfrac{[}{]}{0pt}{}{q}{k}%
\left(  -1\right)  ^{k-1}B_{q+k+1}\\
&  \quad=\frac{q+2}{q\left(  q+1\right)  }\left\{
\begin{array}
[c]{c}%
-%
\genfrac{[}{]}{0pt}{}{q}{0}%
B_{q+1}+%
\genfrac{[}{]}{0pt}{}{q}{1}%
B_{q+2}+%
\genfrac{[}{]}{0pt}{}{q}{q}%
B_{2q+1}-%
\genfrac{[}{]}{0pt}{}{q}{q-1}%
B_{2q}\\
+%
\genfrac{[}{]}{0pt}{}{q}{q-2}%
B_{2q-1}-%
\genfrac{[}{]}{0pt}{}{q}{q-3}%
B_{2q-2}+\sum_{k=2}^{q-4}%
\genfrac{[}{]}{0pt}{}{q}{k}%
\left(  -1\right)  ^{k-1}B_{q+k+1}%
\end{array}
\right\} \\
&  \quad=\frac{q+2}{q\left(  q+1\right)  }\left\{  -%
\genfrac{[}{]}{0pt}{}{q}{q-1}%
B_{2q}-%
\genfrac{[}{]}{0pt}{}{q}{q-3}%
B_{2q-2}+\sum_{k=2}^{q-4}%
\genfrac{[}{]}{0pt}{}{q}{k}%
\left(  -1\right)  ^{k-1}B_{q+k+1}\right\}  .
\end{align*}
By a similar argument as in the proof of Theorem \ref{teo3}, we see that
$\left(  q-1\right)  \nmid\left(  q+k+1\right)  $ for $k=2,3,\ldots q-4,$ so
the sum on the right vanishes modulo $q$. Further, $B_{2q}\equiv0$ $\left(
\operatorname{mod}q\right)  $ by the von Staudt-Clausen theorem. Thus%
\begin{align*}
\frac{q+2}{q\left(  q+1\right)  }\sum_{k=0}^{q}%
\genfrac{[}{]}{0pt}{}{q}{k}%
\left(  -1\right)  ^{k-1}B_{q+k+1}  &  \equiv-\frac{q+2}{q\left(  q+1\right)
}%
\genfrac{[}{]}{0pt}{}{q}{q-3}%
B_{2q-2}\\
&  \equiv-\frac{q+2}{q\left(  q+1\right)  }\frac{q\left(  q-1\right)  q\left(
q-1\right)  \left(  q-2\right)  \left(  q-3\right)  }{48}B_{2q-2}\\
&  \equiv\frac{1}{12}\text{ }\left(  \operatorname{mod}q\right)  .
\end{align*}
Therefore%
\[
B_{q,q+1}\equiv\frac{1}{12}\text{ }\left(  \operatorname{mod}q\right)
\]
as desired.
\end{proof}

\end{document}